\documentclass[12pt,a4paper,reqno]{amsart}
\usepackage[english]{babel}
\usepackage[latin1]{inputenc}
\usepackage{graphicx,tikz}
\usepackage{amssymb, amsmath}
\usepackage{geometry, enumerate,stackrel,mathtools}
\usepackage{enumitem,centernot, color}
\usepackage[colorlinks=true,linkcolor=blue,urlcolor=blue,citecolor=blue]{hyperref}
\usepackage{multirow}
\usepackage{mathtools}

\usepackage{tabularx}
\usepackage{booktabs}
\usepackage{array}

\newcolumntype{C}{>{\centering\arraybackslash}X}

\newtheorem{theorem}{Theorem}[section]

\newtheorem{lemma}[theorem]{Lemma}
\newtheorem{corollary}[theorem]{Corollary}
\newtheorem{proposition}[theorem]{Proposition}

\theoremstyle{definition}

\newtheorem{remark}[theorem]{Remark}

\newcommand{\ds}{\mathrm{\,d}s}
\newcommand{\dt}{\mathrm{\,d}t}
\newcommand{\dx}{\mathrm{\,d}x}

\begin{document}

\title[A conjecture on the norm of the $C^*-I$ for monotone sequences]{A solution to a conjecture on the norm of the Copson operator minus the Identity for monotone sequences}

\author[A. Ben Said]{Achraf Ben Said$^{*}$}

\address{Department of Analysis and Applied Mathematics, Complutense University of Madrid, 28040 Madrid, Spain.}
\email{achbensa@ucm.es}

\author[A. Gogatishvili]{Amiran Gogatishvili$^{**}$}

\address{Institute of Mathematics of the Czech Academy of Sciences, 115 67 Praha 1, Czech Republic.}
\email{gogatish@math.cas.cz}



\begin{abstract}
We investigate the operator distance between the Identity and the Copson operator on the cone of nonnegative, nonincreasing sequences within the Lebesgue sequence spaces $\ell^p(\mathbb{N})$. By proving the exact value of this distance, we confirm a conjecture originally formulated by A. Ben Said, S. Boza, and J. Soria in ``The norm of Hardy-type oscillation operators in the discrete and continuous settings'' (\emph{Analysis and Applications}, 2025).
\end{abstract}
\maketitle

\section{Introduction}
For a given sequence $\{x(n)\}_{n\geq 1}$ of real numbers, the Ces\`aro operator $C$ and its dual, also known as the Copson operator, $C^*$, are defined as
$$Cx(n)=\frac{1}{n} \sum_{k=1}^n x(k) \ \ \ \ {\rm and}\ \ \ \ C^*x(n)=\sum_{k=n}^{\infty} \frac{x(k)}{k},$$
respectively. The continuous counterparts of these operators are the classical Hardy average operator $H$ and the adjoint operator $H^*$. Both operators, $H$ and $H^*$, are given by 
 $$ Hf(t)=\frac1t\int_0^tf(s)\ds \ \ \ \ \ {\rm and}\ \ \ \ \ H^*f(t)=\int_t^\infty\frac{f(s)}{s}\ds,$$
assuming the integral expressions above converge for any such function~$f$ on~$(0,\infty)$. The boundedness of the discrete operators $C$ and $C^*$ on $\ell^p(\mathbb{N})$, as well as the continuous operators $H$ and $H^*$ on $L^p(\mathbb{R}^+)$, follows from the classical Hardy inequalities~\cite{HLP}.

\medskip

G. Bennett formulated in \cite{B} the question of determining the exact norm of $C-I$ as a map acting in $\ell^p(\mathbb{N})$, while he was interested in finding out the best constant for some equivalent norms for the Ces\`{a}ro sequence space $\mathrm{ces}_p(\mathbb{N})$. By means of Hilbert space methods, the operator norm was previously shown to be $1$ for $\ell^2(\mathbb{N})$ (see \cite{BHS}). G.~J.~O.~Jameson \cite{Jameson} conjectured that $\|C-I\|_{\ell^p(\mathbb{N})}=\frac{1}{p-1}$, for $1<p\le 2$ and he obtained the result for $p=4/3$. Very recently, G. Sinnamon \cite{Si} used Jameson's method to prove the conjecture and, in fact, answered Bennett's question in the full range $1<p\leq \infty$ (with a different value of the norm for $2<p\le\infty$). Specifically, he proved that if $m_p=\underset{t\in[0,1/2]}{\min}\big(pt^{p-1}+(1-t)^p-t^p\big)$, then
\begin{equation*}
\|C-I\|_{\ell^p(\mathbb{N})} = 
\begin{cases}
\frac{1}{p-1}, & \text{if } 1 < p \leq 2, \\[1ex]
m_p^{-1/p}, & \text{if } 2 \leq p < \infty, \\[1ex]
2, & \text{if } p = \infty.
\end{cases}
\end{equation*}
The operator norm of $C^* - I$ is obtained via a duality argument;
\begin{equation*}
\|C^*-I\|_{\ell^p(\mathbb{N})} = 
\begin{cases}
2, & \text{if } p=1, \\[1ex]
m_{p'}^{-1/p'}, & \text{if } 1 < p \leq 2, \\[1ex]
p-1, & \text{if } 2\leq p<\infty,
\end{cases}
\end{equation*}
where $p'=p/(p-1)$ is the conjugate exponent of $p$. In \cite{ABS}, the authors removed the dependency on $p'$ for the norm of $C^* - I$ when $1 < p \le 2$ by establishing that $m_{p'}^{-1/p'} = M_p^{1/p}$ in this range, where $M_p = \max_{t \in [0,1/2]} \big( pt^{p-1} + (1-t)^p - t^p \big)$.  

\medskip

Regarding the cone of nonnegative sequences in $\ell^p(\mathbb{N})$, the authors in \cite{ABS} proved the following results:
\begin{theorem} \cite[Remark 2.1, Theorem 2.4, and Proposition 2.5]{ABS} Let $1<p \leq \infty$ and let $x \in \ell^p(\mathbb{N})$ be a nonnegative sequence. If $1<p\leq 2$, then
\begin{equation}
\label{ABS_Eq1}
\|(C-I)x\|_p \leq (p-1)^{-1} \|x\|_p,
\end{equation}
and if $2 \leq p \leq \infty$, then
\begin{equation}
\label{ABS_Eq2}
\|(C-I)x\|_p \leq  \|x\|_p.
\end{equation}
Moreover, both inequalities \eqref{ABS_Eq1} and \eqref{ABS_Eq2} are sharp. 
\end{theorem}

\begin{theorem} \cite[Theorem 3.3 and Proposition 3.6]{ABS} Let $1\leq p < \infty$ and $x\in \ell^p(\mathbb{N})$ be a nonnegative sequence. If $1\leq p\leq 2$, then
\begin{equation}
\label{ABS_Eq3}
\|(C^*-I)x\|_p \leq M_p^{1/p} \|x\|_p,
\end{equation}
and if $2 \leq p < \infty$, then
\begin{equation}
\label{ABS_Eq4}
\|(C^*-I)x\|_p \leq (p-1) \|x\|_p.
\end{equation}
Furthermore, the constants in \eqref{ABS_Eq3} and \eqref{ABS_Eq4} are the best possible.  
\end{theorem}

On the cone of nonnegative, nonincreasing sequences in $\ell^p(\mathbb{N})$, we have the following result for the operator $C - I$:
\begin{theorem} \cite[Theorem 2.2 and Proposition 2.4]{BJ} Let $1<p \leq \infty$ and $x \in \ell^p(\mathbb{N})$ be a nonnegative, nonincreasing sequence. If $1<p\leq 2$, then
\begin{equation}
\label{BJ_Eq1}
\|(C-I)x\|_p \leq (p-1)^{-1} \|x\|_p,
\end{equation}
 if $2 \leq p<  \infty$, then
\begin{equation}
\label{BJ_Eq2}
\|(C-I)x\|_p \leq (p-1)^{-1/p} \|x\|_p,
\end{equation}
and if $p=\infty$, then 
\begin{equation}
\label{BJ_Eq3}
\|(C-I)x\|_\infty \leq  \|x\|_\infty,
\end{equation}
Moreover, the inequalities \eqref{BJ_Eq1}, \eqref{BJ_Eq2}, and \eqref{BJ_Eq3} are sharp. 
\end{theorem}

For the operator $C^*-I$ acting on the cone of nonnegative, nonincreasing sequences in $\ell^p(\mathbb{N})$, we have the following theorem:
\begin{theorem}
\label{TheoremABS}
 \cite[Proposition 3.6 and Theorem 3.8]{ABS} Let $1\leq p < \infty$ and $x\in \ell^p(\mathbb{N})$ be a nonnegative sequence. If $p=1$, then
\begin{equation}
\label{ABS_Eq5}
\|(C^*-I)x\|_1 \leq 2e^{-1} \|x\|_1,
\end{equation}
and if $2 \leq p < \infty$, then
\begin{equation}
\label{ABS_Eq6}
\|(C^*-I)x\|_p \leq (p-1) \|x\|_p.
\end{equation}
Furthermore, the constants in \eqref{ABS_Eq5} and \eqref{ABS_Eq6} are optimal.  
\end{theorem}

\medskip

Further results on optimal constants for $C-I$, its dual $C^*-I$, and their continuous analogues can be found in \cite{Ash, BoSo, Ko, Ko2, Krug}.

\medskip

\medskip

Note that for $1 < p < 2$, the exact norm of $C^* - I$ restricted to the cone of nonnegative, nonincreasing sequences in $\ell^p(\mathbb{N})$, denoted $\|C^* - I\|_{\ell^p_{\mathrm{dec}}(\mathbb{N})}$, has remained open until now. Regarding this constant, Ben Said, Boza, and Soria \cite{ABS} conjectured that
\begin{equation}
\label{conj:norm_Cstar}
\|C^* - I\|_{\ell^p_{\mathrm{dec}}(\mathbb{N})}  = \left( \int_0^1 |\ln t + 1|^p \, \mathrm{d}t \right)^{1/p}\overset{\rm def}{=}C_p^{1/p} \quad \text{for } 1 < p < 2.
\end{equation}
In the same work \cite{ABS}, the authors posed an analogous conjecture for the dual of the continuous Hardy operator minus the Identity. This conjecture was recently confirmed by Ben Said and Sinnamon \cite{AS}, who completed the work initiated in \cite{ABS} by proving the following:

\begin{theorem} \cite[Theorem~1.2]{AS}, \cite[Proposition~4.4 and Theorem~4.5]{ABS}
\label{AS_Theorem} \\Let $1\le p<\infty$. Let $f$ be a nonnegative and nonincreasing function. If $1\leq p \leq 2$, then
\begin{equation}
\label{U1}
(p-1)\|f\|_p \le \|(H^{*}-I)f\|_p \le C_p^{1/p}\|f\|_p,
\end{equation}
and if $2\leq p<\infty$, then
\begin{equation}
\label{U2}
C_p^{1/p}\|f\|_p \le \|(H^{*}-I)f\|_p \le (p-1)\|f\|_p.
\end{equation}
Moreover, the constants $p-1$ and $C_p^{1/p}$ are optimal in both \eqref{U1} and \eqref{U2}.
\end{theorem}

\medskip
Our main result, Theorem~\ref{Main_Theorem}, confirms the validity of the conjecture in \eqref{conj:norm_Cstar}. Together, Theorem~\ref{Main_Theorem} and Theorem~\ref{TheoremABS} yield the same upper constants for $C^*-I$ as those obtained for $H^*-I$ in Theorem~\ref{AS_Theorem}. However, we cannot obtain lower bounds for $C^*-I$ analogous to those for $H^*-I$ in Theorem~\ref{TheoremABS}, as no such finite constants exist in this case (see \cite{ABS2}).

\medskip

This paper is organized into two sections. In Section \ref{Preliminaries}, we introduce several preliminary lemmas that will be used to establish our core results. Section \ref{MainResults} presents the main results of this paper and several noteworthy corollaries.

\medskip

In what follows, for $m,n \in \mathbb{N}$, we adopt the convention that $\sum_{j=m}^n a(j) = 0$ whenever $n < m$.

\section{Preliminaries}
\label{Preliminaries}
In this section, we collect several preliminary results required for the proofs of our main theorems.

\medskip
We start with the following lemma, which will help establish the optimality of our target constant $C_p^{1/p}$.
\begin{lemma}
\label{lemma1}
Let $1 \le p < \infty$. For each $m \in \mathbb{N}$, let $e_m$ denote the $m$-th canonical basis vector of $\ell^p(\mathbb{N})$. Then,
\begin{equation}\label{eq:limit_Cstar}
\lim_{m \to \infty} \frac{\|(C^*-I)C^*e_m\|_p^p}{\|C^*e_m\|_p^p} = C_p.
\end{equation}
\begin{proof}
Let $1\leq p< \infty$ and let $m$ be a natural number. We have that
$$
C^*e_m(n)=\sum_{k=n}^\infty \frac{e_m(k)}{k}= \begin{cases}
\frac{1}{m}, & \text{if } n\leq m,\\
0,  & \text{if } n>m.
\end{cases}
$$
It is clear that $y=C^*e_m$ is nonnegative, nonincreasing sequence and it has $p$-norm $\|C^*e_m\|_p^p=m^{1-p}$, for all $m\in \mathbb{N}$. By applying $C^*$ to $y$ we obtain
$$
C^{*}y=C^*(C^*e_m(n))=\sum_{k=n}^\infty \frac{C^*e_m(k)}{k}= \begin{cases}
\frac{1}{m}\sum_{k=n}^m \frac{1}{k}, & \text{if } n\leq m,\\
0,  & \text{if } n>m.
\end{cases}
$$
So
\begin{equation}
\label{(C*-I)e_m}
|(C^*-I)C^*e_m(n)|= \begin{cases}
\frac{1}{m}\big|-1+\sum_{k=n}^m\frac{1}{k}\big| , & \text{if } n\leq m,\\
0,  & \text{if } n>m.
\end{cases}
\end{equation}
Therefore, we have to show that
\begin{equation}
\label{LIMIT}
\lim_{m \to \infty} \frac{\|(C^*-I)C^*e_m\|_p^p}{\|C^*e_m\|_p^p} = \lim_{m \to \infty} \frac{1}{m} \sum_{n=1}^m \left| -1 + \sum_{k=n}^m \frac{1}{k} \right|^p = C_p^p.
\end{equation}
Let $\{f_m\}_{ m\geq 1}$ be the sequence of functions defined on $(0,1]$ by 
$$f_m(x)=\sum_{n=1}^m \bigg(-1+\sum_{k= \lceil mx \rceil}^m \frac{1}{k} \bigg) \chi_{(\frac{n-1}{m}, \frac{n}{m}]}(x),$$
where $t \mapsto \lceil t \rceil$ is the ceiling function. Note that
\begin{align*}
\int_0^1 |f_m(x)|^p \dx = \sum_{n=1}^m \int_{\frac{n-1}{m}}^{\frac{n}{m}} \bigg | -1+\sum_{k= \lceil mx \rceil}^m \frac{1}{k} \bigg | ^p \dx = \frac{1}{m} \sum_{n=1}^m \left| -1 + \sum_{k=n}^m \frac{1}{k} \right|^p.
\end{align*}
We now study the pointwise convergence of the sequence $\{f_m\}_{m \ge 1}$. Let $x \in (0,1]$. Then there exists $n \in \{1,\ldots, m\}$ such that $x \in \big( \frac{n-1}{m}, \frac{n}{m} \big]$. In this case, we have
\begin{equation*}
f_m(x) = -1 + \sum_{k= \lceil mx \rceil}^m \frac{1}{k} = -1 + H_m - H_{\lceil mx \rceil} + \frac{1}{\lceil mx \rceil},
\end{equation*}
where $H_n = \sum_{k=1}^n \frac{1}{k}$ denotes the $n$-th harmonic number. It is known (see, e.g., \cite{Knuth}) that
\begin{equation}
\label{harmonic}
H_n = \ln n + \gamma + \frac{1}{2n} - \varepsilon_n,
\end{equation}
where $\gamma$ is the Euler--Mascheroni constant and $0 \le \varepsilon_n \le \frac{1}{8n^2}$, which tends to $0$ as $n \to \infty$. Applying \eqref{harmonic}, we obtain
\begin{equation*}
f_m(x) = -1 + \ln \left( \frac{m}{\lceil mx \rceil} \right) + \frac{1}{2\lceil mx \rceil} + \frac{1}{2m} - \varepsilon_m + \varepsilon_{\lceil mx \rceil}.
\end{equation*}
Taking the limit as $m \to \infty$ in the expression above and using the fact that $\lim_{m \to \infty} \ln \big( \frac{m}{\lceil mx \rceil} \big) = -\ln(x)$, we conclude that
\begin{equation*}
\lim_{m\to \infty}f_m(x)=-1-\ln(x) \overset{\rm def}{=} f(x),
\end{equation*}
for all $x \in (0,1]$. We now seek a nonnegative function $g$ such that $|f_m(x)|^p \le g(x)$ for all $x \in (0,1]$ and $m \in \mathbb{N}$. Fixing $x \in (0,1]$, there exists $n \in \{1,\ldots, m\}$ such that $x \in \big( \frac{n-1}{m}, \frac{n}{m} \big]$, which implies $\lceil mx \rceil = n$. We have 
\begin{equation*}
\sum_{k=\lceil mx \rceil}^m \frac{1}{k} = \sum_{k=n}^m \frac{1}{k} = \frac{1}{n} + \sum_{k=n+1}^m \frac{1}{k} \le \frac{1}{n} + \sum_{k=n+1}^m \int_{k-1}^{k} \frac{\mathrm{d}t}{t} = \frac{1}{n} + \int_{n}^{m} \frac{\mathrm{d}t}{t} = \frac{1}{n} + \ln \left( \frac{m}{n} \right).
\end{equation*}
Since $x \le \frac{n}{m}$, it follows that $\frac{m}{n} \le \frac{1}{x}$, and hence $\ln \big( \frac{m}{n} \big) \le \ln \big( \frac{1}{x} \big)$. Therefore,
\begin{equation*}
\sum_{k=\lceil mx \rceil}^m \frac{1}{k} \le \frac{1}{n} + \ln \left( \frac{m}{n} \right) \le 1 - \ln x.
\end{equation*}
Therefore,
\begin{equation*}
f_m(x) = -1 + \sum_{k= \lceil mx \rceil}^m \frac{1}{k} \leq -\ln x \leq 1-\ln x.
\end{equation*}
On the other hand, since $n \le m$, we have
\begin{equation*}
f_m(x) = -1 + \sum_{k= \lceil mx \rceil}^m \frac{1}{k} = -1 + \sum_{k= n}^m \frac{1}{k} \ge -1 \ge -1 + \ln x.
\end{equation*}
Therefore, $|f_m(x)| \le 1 - \ln x$ for all $x \in (0,1]$ and $m \in \mathbb{N}$. Hence,
\begin{equation*}
|f_m(x)|^p \le (1-\ln x)^p = g(x)
\end{equation*}
for all $x \in (0,1]$ and $m \in \mathbb{N}$. Furthermore,
\begin{align*}
\int_0^1 g(x) \, \mathrm{d}x &= \int_0^1 (1-\ln x)^p \, \mathrm{d}x = \int_0^1 \left( \ln \frac{\mathrm{e}}{x} \right)^p \, \mathrm{d}x = \mathrm{e} \int_0^{1/\mathrm{e}} \left( \ln \frac{1}{u} \right)^p \, \mathrm{d}u \\
&< \mathrm{e} \int_0^1 \left( \ln \frac{1}{u} \right)^p \, \mathrm{d}u = \mathrm{e} \, \Gamma(p+1) < \infty,
\end{align*}
where $\Gamma$ denotes Euler's Gamma function. Thus, by the Lebesgue Dominated Convergence Theorem, we obtain
\begin{align*}
\lim_{m \to \infty} \frac{1}{m} \sum_{n=1}^m \left| -1 + \sum_{k=n}^m \frac{1}{k} \right|^p &= \lim_{m \to \infty} \int_0^1 |f_m(x)|^p \, \mathrm{d}x = \int_0^1 \lim_{m \to \infty} |f_m(x)|^p \, \mathrm{d}x \\ 
&= \int_0^1 |f(x)|^p \, \mathrm{d}x = \int_0^1 |1+\ln x|^p \, \mathrm{d}x = C_p.
\end{align*}
This completes the proof of \eqref{LIMIT}.
\end{proof}
\end{lemma}

We define the auxiliary vectors $\rho, d \in \mathbb{R}^m$ by
\begin{equation*}
    \rho(k) = 
    \begin{cases}
        0, & \text{if } k = 1, \\[1.5ex]
        \left( k \ln \!\left( \dfrac{k}{k-1} \right) \right)^{-1}\!, & \text{if } 1 < k \le m,
    \end{cases}
\end{equation*}
and
\begin{equation*}
   d(k) = 
    \begin{cases}
        \displaystyle\sum_{i=k}^{m-1} \frac{\rho(i+1)-\rho(i)}{i}, & \text{if } 1 \leq k \le m, \\[2.5ex]
         0, & \text{if } k = m.
    \end{cases}
\end{equation*}

The next result establishes several essential properties of the vectors $d$ and $\rho$.

\begin{lemma}
\label{Lema_d_rho}
Let $m \in \mathbb{N}$. The vectors $\rho$ and $d$ defined above satisfy the following properties:
\begin{enumerate}
    \item[(i)] The vector $\rho$ is nonnegative, strictly increasing, and bounded; explicitly,
    \begin{equation}
    \label{Lim_Rho=1}
        0 \le \rho(j-1) \le \rho(j) < 1, \quad \text{for all } 1 < j \le m.
    \end{equation}
    \item[(ii)] The vector $d$ is nonnegative and strictly decreasing, and satisfies the identity
    \begin{equation}
    \label{Sum_Vector_d}
         \sum_{k=1}^j d(k)=j d(j) + \rho(j),  \quad \text{for all } 1 \le j \le m.
    \end{equation}
\end{enumerate}
\begin{proof}
We begin with the proof of (i). Define $f \colon (1,\infty) \to \mathbb{R}$ by 
\begin{equation*}
    f(x) = x \ln \left( 1 + \frac{1}{x-1} \right).
\end{equation*}
Differentiating $f$ yields
\begin{equation*}
    f'(x) = \ln \left( 1 + \frac{1}{x-1} \right) - \frac{1}{x-1} < 0 \quad \text{for all } x > 1,
\end{equation*}
where the inequality follows from the standard estimate $\ln(1+t) < t$ for $t > 0$. Thus, $f$ is strictly decreasing on $(1,\infty)$, which implies that $1/f$ is strictly increasing. Consequently, the vector $\rho$ is strictly increasing. 

On the other hand, we have
\begin{align*}
\lim_{x \to \infty} f(x) &= \lim_{x \to \infty} x \ln \left( \frac{x}{x-1} \right) 
= \ln \left( \lim_{x \to \infty} \left( 1 + \frac{1}{x-1} \right)^x \right) 
= \ln(e) = 1.
\end{align*}
Since $f$ is strictly decreasing on $(1,\infty)$ with limit $1$ at infinity, we get $1 < f(x) < \infty$ for all $x \in (1,\infty)$. Consequently, $0 < 1/f(x) < 1$ on $(1,\infty)$, which yields
\begin{equation*}
  \rho(j) < 1 \quad \text{for all } j \in \mathbb{N}.
\end{equation*}

We now turn to the proof of (ii). Let $k \in \{1, \dots, m-1\}$. It is clear from the  definition of $d$ that
\begin{equation}
\label{Identity_dkdk+1_0}
    d(k) - d(k+1) = \frac{\rho(k+1) - \rho(k)}{k}.
\end{equation}
Since $\rho$ is strictly increasing, it follows immediately that $d$ is strictly decreasing, which gives
\begin{equation*}
    0 = d(m) < d(m-1) < \dots < d(2) < d(1).
\end{equation*}
To complete the proof of (ii), it remains to establish the identity
\begin{equation*}
    \sum_{k=1}^j d(k) = j d(j) + \rho(j) \quad \text{for all } 1 \le j \le m.
\end{equation*}
We proceed by induction on $j$. The base case $j=1$ holds trivially. Assume the identity holds for some $1 \le j < m$. For $j+1$, we have
\begin{align*}
    (j+1)d(j+1) + \rho(j+1) &= d(j+1) +  j d(j+1) + \rho(j+1)\\
    &= d(j+1) +  j d(j) + \rho(j) \\ 
    &= d(j+1) + \sum_{k=1}^j d(k) = \sum_{k=1}^{j+1} d(k),
\end{align*}
where the second equality follows from rewriting $d(j) - d(j+1) = \frac{\rho(j+1) - \rho(j)}{j}$ as 
\begin{equation}
\label{Identity_dkdk+1}
j d(j+1) + \rho(j+1) = j d(j) + \rho(j),
\end{equation}
and the third equality applies the induction hypothesis. This completes the proof.
\end{proof}
\end{lemma}

In addition to the vectors $\rho$ and $d$ introduced previously, the following matrix will play a central role in our analysis. For $m \in \mathbb{N}$, consider the square matrix $T = (T_{i,j})_{i,j=1}^{m} \in \mathbb{R}^{m \times m}$ defined by
\begin{equation}
\label{Matrix}
    T_{i,j} = 
    \begin{cases}
        d(\max\{i,j\}), & \text{if } i \neq j, \\[1ex]
        d(i) + \rho(i), & \text{if } i = j.
    \end{cases}
\end{equation}
Explicitly, $T$ has the symmetric form
\medskip
\begin{equation*}
T = \begin{pmatrix}
d(1) + \rho(1) & d(2) & \cdots & d(m-1) & d(m) \\[1ex]
d(2) & d(2) + \rho(2) & \cdots & d(m-1) & d(m) \\[1ex]
d(3) & d(3) & \cdots & d(m-1) & d(m) \\[1ex]
\vdots & \vdots & \ddots & \vdots & \vdots \\[1ex]
d(m-1) & d(m-1) & \cdots & d(m-1) + \rho(m-1) & d(m) \\[1ex]
d(m) & d(m) & \cdots & d(m) & d(m) + \rho(m)
\end{pmatrix}.
\end{equation*}

\medskip

The following lemma establishes crucial properties of the matrix $T$.
\begin{lemma}
For $m \in \mathbb{N}$, the matrix $T$ defined in \eqref{Matrix} is symmetric and non-negative. Moreover, its operator $p$-norm on $\mathbb{R}^m$ satisfies
\begin{equation}
\label{Norm_T}
    \|T\|_p \leq 1.
\end{equation}
\begin{proof}
By definition, $T$ is symmetric. Furthermore, since $d$ and $\rho$ are nonnegative by Lemma~\ref{Lema_d_rho}, all entries of $T$ are nonnegative. Now, let $1 \leq j \leq m$. Then
\begin{align*}
\sum_{i=1}^m T_{i,j} &= \sum_{i=1}^{j-1} T_{i,j} + T_{j,j} + \sum_{i=j+1}^{m} T_{i,j} = \sum_{i=1}^{j-1} d(j) + \big( d(j) + \rho(j)\big) + \sum_{i=j+1}^{m} d(i) \\
&= j d(j) + \rho(j) + \sum_{i=j+1}^{m} d(i) = \sum_{i=1}^{m} d(i) = m d(m) + \rho(m) = \rho(m),
\end{align*}
where we used \eqref{Sum_Vector_d} in the third and fifth equalities. Therefore,
\[
\|T\|_1 = \rho(m) = \|T\|_\infty,
\]
where the second equality follows from the symmetry of $T$. By the Riesz--Thorin interpolation theorem (see, e.g., \cite[Corollary~IV.2.3]{BS}),
\[
\|T\|_p \leq \rho(m) \leq 1
\]
for all $m \geq 1$, where the final inequality follows from \eqref{Lim_Rho=1}. This ends the proof.
\end{proof}
\end{lemma}

The following proposition is of importance.
\begin{proposition} 
\label{Key_Prop}
Let $1\leq p \leq 2$, $m\in \mathbb{N}$ and $b\in \mathbb{R}^m$ a nonnegative vector. Then, the following inequality holds:
\begin{equation}
\label{Ineq_AB}
\sum_{k=1}^m \bigg|\sum_{i=k}^m b(i)\bigg(-1+\sum_{n=k}^i \frac{1}{n}\bigg)\bigg|^p\leq \sum_{k=1}^m \bigg|  \sum_{i=k}^m b(i)\ln\bigg(i\frac{(k-1)^{k-1}}{k^k}\bigg)\bigg|^p.
\end{equation}
\begin{proof}
Fix $1\leq p \leq 2$ and $m \in \mathbb{N}$, and let $b \in \mathbb{R}^m$ be a nonnegative vector. We define the vectors $A, B, \in \mathbb{R}^m$ for each $k \in \{1, \dots, m\}$ by
\begin{align*}
    A(k) &= \sum_{i=k}^m b(i) \ln \left( i \frac{(k-1)^{k-1}}{k^k} \right), \\
    B(k) &= \sum_{i=k}^m b(i) \left( -1 + \sum_{n=k}^i \frac{1}{n} \right).
\end{align*}
To establish inequality \eqref{Ineq_AB}, we have to show that
\begin{equation*}
    \|B\|_p \leq \|A\|_p.
\end{equation*}
To this end, we will demonstrate that $B = T A$, where $T$ is the matrix defined in~\eqref{Matrix}. We have to prove that 
\begin{equation}
\label{B=TA}
B(j)=\sum_{k=1}^m T_{j,k}A(k),
\end{equation}
for all $1 \leq j \leq m$. Let $j \in \{1,\cdots, m\}$. Then
\begin{align*}
\sum_{k=1}^m T_{j,k}A(k)&=\sum_{k=1}^m T_{j,k}\sum_{i=k}^m b(i) \ln \left( i \frac{(k-1)^{k-1}}{k^k} \right) \\
&=\sum_{i=1}^m b(i) \sum_{k=1}^{i} T_{j,k} \ln \left( i \frac{(k-1)^{k-1}}{k^k} \right).
\end{align*}
To prove \eqref{B=TA} it is enough to show that
\begin{equation}
\label{TkjCij}
   \sum_{k=1}^{i} T_{j,k} \ln \left( i \frac{(k-1)^{k-1}}{k^k} \right) = 
    \begin{cases}
        -1 + \sum_{n=j}^i \frac{1}{n}, & \text{if } j \leq i \leq m , \\[1ex]
        0, & \text{if } 1\leq i <j.
    \end{cases}
\end{equation}
Notice that for $1 \leq i < j$, we have
\begin{align*}
\sum_{k=1}^{i} T_{j,k} \ln \left( i \frac{(k-1)^{k-1}}{k^k}\right) 
&= d(j) \sum_{k=1}^{i} \ln \left( i \frac{(k-1)^{k-1}}{k^k} \right) \\
&= d(j) \left( i \ln(i) - \sum_{k=1}^{i} \big( k \ln(k) - (k-1) \ln(k-1) \big) \right) \\
&= d(j) \big( i \ln(i) - i \ln(i) \big) = 0,
\end{align*}
where we have used that $x \ln x=0$ for $x=0$. Then,
for $j \le i \le m$, we obtain
\begin{align*}
\sum_{k=1}^{i} T_{j,k} \ln \left( i \frac{(k-1)^{k-1}}{k^k}\right) 
&= \sum_{k=1}^{j-1} T_{j,k} \ln \left( i \frac{(k-1)^{k-1}}{k^k}\right) + T_{j,j}\ln \left( i \frac{(j-1)^{j-1}}{j^j}\right) \\
&\qquad +\sum_{k=j+1}^{i} T_{j,k} \ln \left( i \frac{(k-1)^{k-1}}{k^k}\right) \\
&=d(j) \sum_{k=1}^{j}\ln \left( i \frac{(k-1)^{k-1}}{k^k}\right)+ \rho(j)\ln \left( i \frac{(j-1)^{j-1}}{j^j}\right) \\
&\qquad +\sum_{k=j+1}^{i} d(k) \ln \left( i \frac{(k-1)^{k-1}}{k^k}\right) \\
&=d(j)j\ln \bigg(\frac{i}{j}\bigg) + \rho(j)\ln \left( i \frac{(j-1)^{j-1}}{j^j}\right) \\
&\qquad+\sum_{k=j+1}^{i} d(k) \ln \left( i \frac{(k-1)^{k-1}}{k^k}\right).
\end{align*}
That is
\begin{align}
\sum_{k=1}^{i} T_{j,k} \ln \left( i \frac{(k-1)^{k-1}}{k^k}\right)& =d(j)j\ln \bigg(\frac{i}{j}\bigg) + \rho(j)\ln \left( i \frac{(j-1)^{j-1}}{j^j}\right) \nonumber \\
&\qquad +\sum_{k=j+1}^{i} d(k) \ln \left( i \frac{(k-1)^{k-1}}{k^k}\right). \label{Identity_Tjk}
\end{align}
We now focus on the third summand on the right-hand side of the expression above.
\begin{align*}
\sum_{k=j+1}^{i} d(k) \ln \left( i \frac{(k-1)^{k-1}}{k^k}\right)&=\ln(i) \sum_{k=j+1}^{i} d(k)+\sum_{k=j+1}^{i} d(k)(k-1)\ln(k-1) \\
&\qquad -\sum_{k=j+1}^{i} d(k)k\ln(k) \\
&=\ln(i) \sum_{k=j+1}^{i} d(k)+\sum_{k=j}^{i-1} d(k+1)k\ln(k) -\sum_{k=j+1}^{i} d(k)k\ln(k)
 \end{align*}
\begin{align*}
&=\ln(i) \sum_{k=j+1}^{i} d(k)+d(j+1)j\ln(j)-d(i)i\ln(i) + \sum_{k=j+1}^{i-1} \big(d(k+1)-d(k)\big)k\ln(k)\\
& =\ln(i) \sum_{k=j+1}^{i} d(k)+d(j+1)j\ln(j)-d(i)i\ln(i)  + \sum_{k=j+1}^{i-1} \big(\rho(k)-\rho(k+1)\big)\ln(k) \\
&=\ln(i) \sum_{k=j+1}^{i} d(k)+d(j+1)j\ln(j)-d(i)i\ln(i) \\
&\qquad + \rho(j+1)\ln(j+1)-\rho(i)\ln(i-1)+ \sum_{k=j+2}^{i-1} \rho(k)\ln\bigg(\frac{k}{k-1}\bigg) \\
&=\ln(i) \sum_{k=j+1}^{i} d(k)+\big(d(j+1)j+\rho(j+1)\big)\ln(j)-\big(d(i)i+\rho(i)\big)\ln(i) \\
&\qquad + \rho(j+1)\ln\bigg(\frac{j+1}{j}\bigg)+\rho(i)\ln\bigg(\frac{i}{i-1}\bigg)+ \sum_{k=j+2}^{i-1} \rho(k)\ln\bigg(\frac{k}{k-1}\bigg) \\
&=\ln(i) \sum_{k=j+1}^{i} d(k)+\big(d(j)j+\rho(j)\big)\ln(j)-\big(d(i)i+\rho(i)\big)\ln(i) \\
&\qquad + \rho(j+1)\ln\bigg(\frac{j+1}{j}\bigg)+\rho(i)\ln\bigg(\frac{i}{i-1}\bigg)+ \sum_{k=j+2}^{i-1} \rho(k)\ln\bigg(\frac{k}{k-1}\bigg) \\
&=\ln(i) \sum_{k=j+1}^{i} d(k)+\ln(j)\sum_{k=1}^{j} d(k)-\ln(i)\sum_{k=1}^{i} d(k)  + \sum_{k=j+1}^{i} \rho(k)\ln\bigg(\frac{k}{k-1}\bigg)\\
&=-\ln\bigg(\frac{i}{j}\bigg)\sum_{k=1}^{j} d(k)  + \sum_{k=j+1}^{i} \rho(k)\ln\bigg(\frac{k}{k-1}\bigg)\\
&=-\ln\bigg(\frac{i}{j}\bigg)\big(jd(j)+\rho(j))  + \sum_{k=j+1}^{i} \rho(k)\ln\bigg(\frac{k}{k-1}\bigg),
\end{align*}
where we used \eqref{Identity_dkdk+1_0} in the fourth equality, \eqref{Identity_dkdk+1} in the seventh equality, and \eqref{Sum_Vector_d} in the eighth equality. To summarize, we have proved the following equality:
$$\sum_{k=j+1}^{i} d(k) \ln \left( i \frac{(k-1)^{k-1}}{k^k}\right)= -\ln\bigg(\frac{i}{j}\bigg)\big(jd(j)+\rho(j))  + \sum_{k=j+1}^{i} \rho(k)\ln\bigg(\frac{k}{k-1}\bigg).$$
Substituting this expression into \eqref{Identity_Tjk}, we obtain
\begin{align*}
\sum_{k=1}^{i} T_{j,k} \ln \left( i \frac{(k-1)^{k-1}}{k^k}\right)& =d(j)j\ln \bigg(\frac{i}{j}\bigg) + \rho(j)\ln \left( i \frac{(j-1)^{j-1}}{j^j}\right) \\
&\qquad +\sum_{k=j+1}^{i} d(k) \ln \left( i \frac{(k-1)^{k-1}}{k^k}\right) 
\end{align*}
\begin{align*}
&=d(j)j\ln \bigg(\frac{i}{j}\bigg) + \rho(j)\ln \left( i \frac{(j-1)^{j-1}}{j^j}\right) \\
&\qquad -\ln\bigg(\frac{i}{j}\bigg)\big(jd(j)+\rho(j))  + \sum_{k=j+1}^{i} \rho(k)\ln\bigg(\frac{k}{k-1}\bigg)\\
&=-\rho(j)(j-1)\ln\bigg(\frac{j}{j-1}\bigg)+  \sum_{k=j+1}^{i} \rho(k)\ln\bigg(\frac{k}{k-1}\bigg) \\
&=-\rho(j)j\ln\bigg(\frac{j}{j-1}\bigg)+  \sum_{k=j}^{i} \rho(k)\ln\bigg(\frac{k}{k-1}\bigg) =-1+\sum_{k=j}^{i} \frac{1}{k}.
\end{align*}
This ends the proof of \eqref{TkjCij}. Consequently,
$$\|B\|_p = \|TA\|_p \leq \|T\|_p \|A\|_p \leq \|A\|_p,$$
where we have used $\|T\|_p\leq 1$ from \eqref{Norm_T}. The proof is complete.
\end{proof}
\end{proposition}

Let $S$ denote the left-shift operator defined by 
$$(Sx)(n) = x(n+1)$$
for any sequence of real numbers $x = \{x(n)\}_{n \ge 1}$. We shall require the following two results to derive Corollary~\ref{Corollary 3.3} from our main result.
\begin{proposition}{\rm (}\cite[Corollary 10.13]{B}, \cite[Theorem 1.6]{ABS2}{\rm )}
\label{Proposition_Bennet} Let $1<p<\infty$ and suppose that $x$ is a nonnegative, nonincreasing sequence. If $1<p\leq 2$, then
\begin{equation}
\label{Ineq7}
\|x\|_p \leq (p-1)^{1/p} \|(C-S)x\|_p,
\end{equation}
and if $2\leq p <\infty$, then 
\begin{equation}
    \label{Ineq7_1}
\|x\|_p \leq (p-1) \|(C-S)x\|_p.
\end{equation}
Moreover, the constants in \eqref{Ineq7} and \eqref{Ineq7_1} are the best possible.
\end{proposition}
 
\begin{lemma}\cite [Lemma 4.7]{B}
\label{lemma3} Let $p > 1$. Then, 
\begin{equation}
\label{Ineq8}
\lim_{n\to \infty} n^{p-1} \sum_{k=n}^{\infty} \frac{1}{k^p}= \frac{1}{p-1}.
\end{equation}
\end{lemma}

\section{Main Results}
\label{MainResults}
 In this section, we present our principal results and their corollaries. 

\medskip

Our main theorem is stated below. Although the cases $p=1$ and $p=2$, including the optimality of the constants $C_1 = 2/e$ and $C_2 = 1$, were established in \cite{ABS}, we include them here because our approach provides a new, unified proof.

\begin{theorem}
\label{Main_Theorem}
Let $1 \leq p \leq 2$ and let $x$ be a nonnegative, nonincreasing sequence. Then
\begin{equation}
\label{Ineq3}
\|(C^*-I)x\|_p \leq C_p^{1/p}\|x\|_p,
\end{equation}
where $C_p^{1/p}$ is defined in \eqref{conj:norm_Cstar}. Moreover, the constant $C_p^{1/p}$ is sharp in \eqref{Ineq3}.
\begin{proof}
Let $1\leq p \leq 2$ and let $x$ be a nonnegative, nonincreasing sequence. If $x \notin \ell^p(\mathbb{N})$ then \eqref{Ineq3} holds trivially. Otherwise, since finitely supported sequences are dense in $\ell^p(\mathbb{N})$, it suffices to prove the result for $x$ with finite support. Then, there is a natural number $m$ such that 
$$x(1)\geq x(2)\geq \cdots \geq x(m)>0.$$
and 
$$x(n)=0,$$
for all $n>m$. Let $f_x:(0,\infty) \to \mathbb{R}$ be the simple function given by 
$$f_x(s)=\sum_{k=1}^m x(k)\chi_{(k-1,k]}(s).$$
It is clear that $f_x$ is a nonnegative, nonincreasing  and $\|f_x\|_{L^p(\mathbb{R}^+)}=\|x\|_{\ell^p}$. Applying inequality \eqref{U1}, we obtain
$$\|(H^*-I)f_x\|_{L^p(\mathbb{R}^+)} \leq C(p)^{1/p}\|f_x\|_{L^p(\mathbb{R}^+)}=C(p)^{1/p}\|x\|_{\ell^p}.$$
Therefore, in order to prove \eqref{Ineq3} it suffices to prove that 
\begin{equation}
\label{Ineq4}
\|(C^*-I)x\|_{\ell^p}\leq \|(H^*-I)f_x\|_{L^p(\mathbb{R}^+)}.
\end{equation}
Now, if $k-1<s\leq k$ we have that
\begin{align*}
H^*f_x(s)&=\int_s^\infty f_x(t) \frac{\dt}{t}=x(k)\int_s^k \frac{\dt}{t}+\sum_{n=k+1}^m x(n)\int_{n-1}^n \frac{\dt}{t}\\
&=x(k)\ln\bigg(\frac{k}{s}\bigg)+\sum_{n=k+1}^m x(n)\ln\bigg(\frac{n}{n-1}\bigg).
\end{align*}
So
\begin{align*}
H^*f_x(s)=\sum_{k=1}^m \bigg(x(k)\ln\bigg(\frac{k}{s}\bigg)+\sum_{n=k+1}^m x(n)\ln\bigg(\frac{n}{n-1}\bigg)\bigg)\chi_{(k-1,k]}(s).
\end{align*}
Hence
\begin{align*}
(H^*-I)f_x(s)=\sum_{k=1}^m \bigg(x(k)\bigg(\ln\bigg(\frac{k}{s}\bigg)-1\bigg)+\sum_{n=k+1}^m x(n)\ln\bigg(\frac{n}{n-1}\bigg)\bigg)\chi_{(k-1,k]}(s).
\end{align*}
Thus
\begin{align*}
\|(H^*-I)f_x\|_{L^p(\mathbb{R}^+)}^p=\sum_{k=1}^m \int_{k-1}^k \bigg| x(k)\bigg(\ln\bigg(\frac{k}{s}\bigg)-1\bigg)+\sum_{n=k+1}^m x(n)\ln\bigg(\frac{n}{n-1}\bigg)\bigg|^p\ds.
\end{align*}
On the other hand, we have that
\begin{align*}
\|(C^*-I)x\|_{\ell^p}^p=\sum_{k=1}^m \bigg|-x(k)+\sum_{n=k}^m \frac{x(n)}{n}\bigg|^p.
\end{align*}
Then, proving the inequality \eqref{Ineq4} is equivalent to proving the following inequality
\begin{align*}
&\sum_{k=1}^m \int_{k-1}^k \bigg| x(k)\bigg(\ln\bigg(\frac{k}{s}\bigg)-1\bigg)+\sum_{n=k+1}^m x(n)\ln\bigg(\frac{n}{n-1}\bigg)\bigg|^p\ds \\
&\qquad \geq \sum_{k=1}^m \bigg|-x(k)+\sum_{n=k}^m \frac{x(n)}{n}\bigg|^p.
\end{align*}
Now, by taking into account that $t \mapsto |t|^p$ is a convex function, we have that
\begin{align*}
&\sum_{k=1}^m \int_{k-1}^k \bigg| x(k)\bigg(\ln\bigg(\frac{k}{s}\bigg)-1\bigg)+\sum_{n=k+1}^m x(n)\ln\bigg(\frac{n}{n-1}\bigg)\bigg|^p\ds \\
&\qquad \geq
\sum_{k=1}^m \bigg|\int_{k-1}^k  \bigg(x(k)\bigg(\ln\bigg(\frac{k}{s}\bigg)-1\bigg)+\sum_{n=k+1}^m x(n)\ln\bigg(\frac{n}{n-1}\bigg)\bigg)\ds\bigg|^p \\
&\qquad = \sum_{k=1}^m \bigg|  x(k)(1-k)\ln\bigg(\frac{k}{k-1}\bigg)+\sum_{n=k+1}^m x(n)\ln\bigg(\frac{n}{n-1}\bigg)\bigg|^p.
\end{align*}
So, to prove \eqref{Ineq4}, it suffices to prove that the following inequality holds
\begin{align}
\sum_{k=1}^m \bigg|  x(k)(1-k)\ln\bigg(\frac{k}{k-1}\bigg)&+\sum_{n=k+1}^m x(n)\ln\bigg(\frac{n}{n-1}\bigg)\bigg|^p\notag \\
&\qquad\geq \sum_{k=1}^m \bigg|-x(k)+\sum_{n=k}^m \frac{x(n)}{n}\bigg|^p.
\label{Ineq5}
\end{align}
Now, since $x$ is a nonnegative, nonincreasing sequence, there is a nonnegative vector $b=\{b(i)\}_{i=1}^m$ such that
$x(k)=\sum_{i=k}^m b(i),$
for all $k=1,\cdots, m$. On the one hand, we have that
\begin{align*}
&\sum_{k=1}^m \bigg|  x(k)(1-k)\ln\bigg(\frac{k}{k-1}\bigg)+\sum_{n=k+1}^m x(n)\ln\bigg(\frac{n}{n-1}\bigg)\bigg|^p  \\
&\qquad =\sum_{k=1}^m \bigg|  \sum_{i=k}^m b(i)(1-k)\ln\bigg(\frac{k}{k-1}\bigg)+\sum_{n=k+1}^m \sum_{i=n}^m b(i)\ln\bigg(\frac{n}{n-1}\bigg)\bigg|^p \\
&\qquad =\sum_{k=1}^m \bigg|  \sum_{i=k}^m b(i)(1-k)\ln\bigg(\frac{k}{k-1}\bigg)+\sum_{i=k+1}^m  b(i) \sum_{n=k+1}^i\ln\bigg(\frac{n}{n-1}\bigg)\bigg|^p \\
&\qquad= \sum_{k=1}^m \bigg|  \sum_{i=k}^m b(i)(1-k)\ln\bigg(\frac{k}{k-1}\bigg)+\sum_{i=k+1}^m  b(i)\ln\bigg(\frac{i}{k}\bigg)\bigg|^p\\
&\qquad= \sum_{k=1}^m \bigg|  \sum_{i=k}^m b(i)(1-k)\ln\bigg(\frac{k}{k-1}\bigg)+\sum_{i=k}^m  b(i)\ln\bigg(\frac{i}{k}\bigg)\bigg|^p \\
&\qquad= \sum_{k=1}^m \bigg|  \sum_{i=k}^m b(i)\bigg(\ln\bigg(\frac{i}{k}\bigg)-(k-1)\ln\bigg(\frac{k}{k-1}\bigg)\bigg)\bigg|^p.
\end{align*}
On the other hand, we have
\begin{align*}
\sum_{k=1}^m \bigg|-x(k)+\sum_{n=k}^m \frac{x(n)}{n}\bigg|^p &= \sum_{k=1}^m \bigg|-\sum_{i=k}^m b(i)+\sum_{n=k}^m \frac{1}{n}\sum_{i=n}^m b(i)\bigg|^p \\
&=\sum_{k=1}^m \bigg|-\sum_{i=k}^m b(i)+\sum_{i=k}^m b(i)\sum_{n=k}^i \frac{1}{n}\bigg|^p \\
&=\sum_{k=1}^m \bigg|\sum_{i=k}^m b(i)\bigg(-1+\sum_{n=k}^i \frac{1}{n}\bigg)\bigg|^p.
\end{align*}
Therefore, the inequality \eqref{Ineq5} is equivalent to 
\begin{align*}
&\sum_{k=1}^m \bigg|  \sum_{i=k}^m b(i)\bigg(\ln\bigg(\frac{i}{k}\bigg)-(k-1)\ln\bigg(\frac{k}{k-1}\bigg)\bigg)\bigg|^p  \geq \sum_{k=1}^m \bigg|\sum_{i=k}^m b(i)\bigg(-1+\sum_{n=k}^i \frac{1}{n}\bigg)\bigg|^p,
\end{align*}
which is the statement of Proposition \ref{Key_Prop}.

\medskip

The optimality of the constant $C_p^{1/p}$ is proven in Lemma \ref{lemma1} by considering the sequence $\{C^*e_m
\}_{m \ge 1}$.
\end{proof}
\end{theorem}

\medskip

\begin{remark}
Note that we can also obtain \eqref{ABS_Eq6} for $2 \le p < \infty$ by an argument analogous to the proof of the previous theorem, together with the corresponding continuous inequality in \eqref{U2}.
\end{remark}

\medskip

Let $C^{*2} = C^* \circ C^*$. Theorem \ref{TheoremABS} together with Theorem~\ref{Main_Theorem} have the following equivalent formulation.

\begin{theorem}
Let $1 \le p <\infty$ and let $x$ be a nonnegative sequence such that $C^*x(n)<\infty$ for all $n\geq 1$. If $1\leq p \leq 2$, then
\begin{equation}
\label{Ineq6}
\|(C^{*2}-C^*)x\|_p \leq C_p^{1/p}\|C^*x\|_p,
\end{equation}
and if $2 \leq p <\infty$, then
\begin{equation}
    \label{Ineq6_1}
    \|(C^{*2}-C^*)x\|_p \leq (p-1)\|C^*x\|_.
\end{equation}
Furthermore, the constants in \eqref{Ineq6} and \eqref{Ineq6_1} are optimal. 
\begin{proof}
It suffices to prove the equivalence between \eqref{Ineq3} and \eqref{Ineq6}, since that between \eqref{ABS_Eq6} and \eqref{Ineq6_1} admits an entirely analogous proof.

 Let $x$ be a nonnegative sequence and set $y = C^*x$. Clearly, $y$ is well defined by hypothesis and is a nonnegative, nonincreasing sequence. Applying \eqref{Ineq3} directly to $y$ yields \eqref{Ineq6}.

\medskip

Conversely, let $x = \{x(k)\}_{k \ge 1}$ be a nonnegative, nonincreasing sequence, and define $y = \{y(k)\}_{k \ge 1}$ by 
\begin{equation}
\label{Def_y}
y(k) = k \big(x(k) - x(k+1)\big), \quad k \ge 1.
\end{equation}
Since $x$ is nonnegative and nonincreasing, $y$ is clearly nonnegative. Applying \eqref{Ineq6} directly to $y$ completes the proof of \eqref{Ineq3}.

\medskip

To establish optimality, let $\{x_n\}_{n \ge 1}$ be a sequence of nonnegative, nonincreasing sequences such that 
\[
\lim_{n \to \infty} \frac{\|(C^*-I)x_n\|_p}{\|x_n\|_p} = C_p^{1/p}.
\]
Considering $\{y_n\}_{n \ge 1}$ defined as in \eqref{Def_y}, we have $C^*y_n = x_n$ for each $n \ge 1$, which gives
\[
\lim_{n \to \infty} \frac{\|(C^*-I)C^*y_n\|_p}{\|C^*y_n\|_p} = \lim_{n \to \infty} \frac{\|(C^*-I)x_n\|_p}{\|x_n\|_p} = C_p^{1/p}.
\]

\medskip

Conversely, to verify optimality from the dual perspective, let $\{y_n\}_{n \ge 1}$ be a sequence of nonnegative sequences satisfying 
\[
\lim_{n \to \infty} \frac{\|(C^*-I)C^*y_n\|_p}{\|C^*y_n\|_p} = C_p^{1/p}.
\]
Defining $x_n = C^*y_n$ for each $n \ge 1$, it follows that
\[
\lim_{n \to \infty} \frac{\|(C^*-I)x_n\|_p}{\|x_n\|_p} = \lim_{n \to \infty} \frac{\|(C^*-I)C^*y_n\|_p}{\|C^*y_n\|_p} = C_p^{1/p},
\]
which completes the proof.
\end{proof}
\end{theorem}

By combining Theorem \ref{TheoremABS}, Proposition~\ref{Proposition_Bennet}, and Theorem~\ref{Main_Theorem}, we obtain the optimal constant in an inequality relating $C-S$ and $C^*-I$ for decreasing sequences.

\begin{corollary}
\label{Corollary 3.3}
Let $1<p<\infty$ and let $x$ be a nonnegative, nonincreasing sequence. If $1<p\leq 2$, then
\begin{equation}
\label{Ineq9}
\|(C^*-I)x\|_p \leq C_p^{1/p}(p-1)^{1/p} \|(C-S)x\|_p,
\end{equation}
and if $2\leq p<\infty$, then
\begin{equation}
\label{Ineq9_1}
\|(C^*-I)x\|_p \leq (p-1)^2 \|(C-S)x\|_p.
\end{equation}
Furthermore, the constants in \eqref{Ineq9} and \eqref{Ineq9_1} are the best possible.
\begin{proof}
We shall prove only \eqref{Ineq9}, as the proof of inequality \eqref{Ineq9_1} is analogous. Let $1 < p \le 2$ and let $x$ be a nonnegative, nonincreasing sequence. If $(C-S)x \notin \ell^p(\mathbb{N})$, then \eqref{Ineq9} holds trivially. Otherwise, taking \eqref{Ineq7} into account, we deduce that the sequence $x \in \ell^p(\mathbb{N})$. Applying \eqref{Ineq7} and \eqref{Ineq3}, we obtain
$$\|(C^*-I)x\|_p \leq C(p)^{1/p} \|x\|_p \leq C(p)^{1/p}(p-1)^{1/p}\|(C-S)x\|_p.$$
To see the optimality of the constant in \eqref{Ineq9} we consider the sequence $\{C^*e_m\}_{m\geq 1}$. We have seen in \eqref{(C*-I)e_m} that 
$$
|(C^*-I)C^*e_m|= \begin{cases}
\frac{1}{m}\big|-1+\sum_{k=n}^m\frac{1}{k}\big| , & \text{if } n\leq m,\\
0,  & \text{if } n>m.
\end{cases}
$$
A straightforward calculation yields
$$
|(C-S)C^*e_m(n)|= \begin{cases}
0 , & \text{if } n< m,\\
\frac{1}{n},  & \text{if } n\geq m.
\end{cases}
$$
\begin{align*}
\lim_{m\to \infty} \frac{\|(C^*-I)C^*e_m\|_p^p}{\|(C-S)C^*e_m\|_p^p}&=\lim_{m\to \infty} \frac{ \frac{1}{m^p}\sum_{n=1}^m \big|-1+\sum_{k=n}^m \frac{1}{k}\big|^p}{\sum_{k=m}^\infty \frac{1}{n^p}}= \\
&=\frac{\underset{m\to \infty}{\lim}\frac{1}{m}\sum_{n=1}^m \big|-1+\sum_{k=n}^m \frac{1}{k}\big|^p}{ \underset{m\to \infty}{\lim} m^{p-1}\sum_{k=m}^\infty \frac{1}{n^p}} =C(p)(p-1),
\end{align*}
where we have used \eqref{LIMIT} and \eqref{Ineq8}.

We now demonstrate the sharpness of the constant in
\eqref{Ineq9_1}. Let $0<\varepsilon<1/p'$, where
$p'$ denotes the conjugate exponent of $p$ given by $\frac{1}{p}+\frac{1}{p'}=1$. Consider the sequence $x_\varepsilon$ defined by
\[
x_\varepsilon(n)=\frac{1}{n^{1/p+\varepsilon}},
\qquad n\geq 1.
\]
Clearly, $x_\varepsilon\in\ell^p(\mathbb N)$, since
\[
\|x_\varepsilon\|_p^p
=\zeta(1+p\varepsilon)<\infty,
\]
where $\zeta$ denotes the Riemann zeta function.

Since $t\mapsto t^{-1-1/p-\varepsilon}$
is decreasing, we have
\begin{align*}
(C^*x_\varepsilon)(n)
&=\sum_{k=n}^\infty
\frac{1}{k^{1+1/p+\varepsilon}}\geq
\int_n^\infty
\frac{dt}{t^{1+1/p+\varepsilon}}=\frac{1}{1/p+\varepsilon}
\frac{1}{n^{1/p+\varepsilon}}.
\end{align*}
Consequently,
\[
(C^*-I)x_\varepsilon(n)
\geq
\left(\frac{1}{1/p+\varepsilon}-1\right)
x_\varepsilon(n),
\qquad n\geq 1.
\]
Since $0<\varepsilon<1/p'$, the coefficient
on the right-hand side is positive.
Taking $\ell^p$-norms and using
\eqref{ABS_Eq6}, we obtain
\[
\frac{1}{1/p+\varepsilon}-1
\leq
\frac{\|(C^*-I)x_\varepsilon\|_p}
{\|x_\varepsilon\|_p}
\leq p-1.
\]
Letting $\varepsilon\to0^+$ and applying
the squeeze theorem yields
\begin{equation}
\label{LimitX_1}
\lim_{\varepsilon\to0^+}
\frac{\|(C^*-I)x_\varepsilon\|_p}
{\|x_\varepsilon\|_p}
=p-1.
\end{equation}

On the other hand, since $0<\varepsilon<1/p'$, we have
\begin{align*}
(Cx_\varepsilon)(n)
&=\frac1n\sum_{k=1}^n
\frac1{k^{1/p+\varepsilon}}\leq
\frac1n\sum_{k=1}^n
\int_{k-1}^k
\frac{dt}{t^{1/p+\varepsilon}}=\frac1n\int_0^n
t^{-1/p-\varepsilon}\,dt\\
&=\frac{1}{1-1/p-\varepsilon}
\frac1{n^{1/p+\varepsilon}},
\end{align*}
for all $n\geq1$. Consequently,
\begin{align*}
((C-S)x_\varepsilon)(n)
&\leq
\frac{1}{1-1/p-\varepsilon}
\frac1{n^{1/p+\varepsilon}}
-\frac1{(n+1)^{1/p+\varepsilon}}\\
&=
\left(
\frac{1}{1-1/p-\varepsilon}-1
\right)x_\varepsilon(n)+
\frac1{n^{1/p+\varepsilon}}
-\frac1{(n+1)^{1/p+\varepsilon}}\\
&\leq
\left(
\frac{1}{1-1/p-\varepsilon}-1
\right)x_\varepsilon(n)+
\left(\frac1p+\varepsilon\right)
\frac1{n^{1+1/p+\varepsilon}},
\end{align*}
where the last inequality follows from the
mean value theorem.

Since $x_\varepsilon$ is nonincreasing,
$(C-S)x_\varepsilon(n)\geq0$ for every $n\geq1$.
Taking $\ell^p$-norms, applying Minkowski's
inequality, and using \eqref{Ineq7_1}, we obtain
\begin{align*}
\frac1{p-1}
&\leq
\frac{\|(C-S)x_\varepsilon\|_p}
{\|x_\varepsilon\|_p}\leq
\frac{1}{1-1/p-\varepsilon}-1+
\left(\frac1p+\varepsilon\right)
\left(
\frac{\zeta(p+1+p\varepsilon)}
{\zeta(1+p\varepsilon)}
\right)^{1/p}.
\end{align*}
Since
\[
\lim_{\varepsilon\to0^+}
\left(
\frac{1}{1-1/p-\varepsilon}-1
\right)
=\frac1{p-1},
\]
and
\[
\lim_{\varepsilon\to0^+}
\frac{\zeta(p+1+p\varepsilon)}
{\zeta(1+p\varepsilon)}
=0,
\]
the squeeze theorem yields
\begin{equation}
\label{LimitX_2}
\lim_{\varepsilon\to0^+}
\frac{\|(C-S)x_\varepsilon\|_p}
{\|x_\varepsilon\|_p}
=\frac1{p-1}.
\end{equation}

Finally, combining \eqref{LimitX_1} and
\eqref{LimitX_2}, we obtain
\begin{align*}
\lim_{\varepsilon\to0^+}
\frac{\|(C^*-I)x_\varepsilon\|_p}
{\|(C-S)x_\varepsilon\|_p}
&=
\frac{p-1}{1/(p-1)}=(p-1)^2.
\end{align*}
Therefore, the constant $(p-1)^2$ in
\eqref{Ineq9_1} is optimal.

\end{proof}  
\end{corollary}

\medskip

\textbf{Acknowledgements:} The first author extends his sincere gratitude to his PhD supervisors, Santiago Boza and Javier Soria, for their continuous guidance and support throughout his doctoral studies. Additionally, he thanks the Institute of Mathematics of the Czech Academy of Sciences for hosting his research visit, during which part of this work was conducted.

\medskip

\textbf{Funding:} The first  author was partially supported by grants PID2020-113048GB-I00 and PID2024-155917NB-I00, funded by MCIN/AEI/10.13039/501100011033. The second author was partially supported by RVO: 67985840, Institute of Mathematics of the Czech Academy of Sciences.

\medskip

\textbf{Data availability:} No data were used for the research described in the article.

\medskip

\textbf{Declaration of generative AI use:}
The author reports generative AI was not used in his research or preparation of this manuscript.

\medskip

\textbf{Conflicts of Interest:} The authors declare that they have no conflict of interest.

\end{document}